\documentclass[11pt,a4paper]{amsart}
\usepackage[all]{xy}
\usepackage{amssymb,mathrsfs,euscript,enumitem,amsmath,bbold}

\usepackage[in]{fullpage}
\usepackage{color}
\definecolor{Chocolat}{rgb}{0.36, 0.2, 0.09}
\definecolor{BleuTresFonce}{rgb}{0.215, 0.215, 0.36}
\definecolor{EgyptianBlue}{rgb}{0.06, 0.2, 0.65}

\usepackage[colorlinks,final,hyperindex]{hyperref}
\hypersetup{citecolor=BleuTresFonce, urlcolor=EgyptianBlue, linkcolor=Chocolat}

\usepackage{fourier}

\catcode`\@=11

\catcode`\@=13

\newtheorem{theorem}{Theorem}
\newtheorem{corollary}[theorem]{Corollary}

\newtheorem{lemma}[theorem]{Lemma}
\newtheorem{proposition}[theorem]{Proposition}

\theoremstyle{definition}
\newtheorem{example}[theorem]{Example}
\newtheorem{remark}[theorem]{Remark}
\newtheorem{definition}[theorem]{Definition}

\def\reg{{\rm reg}}\def\sng{{\rm sng}}
\def\tlustatecka{ \blacksquare}
\def\tecka{\hbox {\Large $\bullet$}}
\def\edg{{\rm edg}}
\def\mn{\mu^{(n+1)}}

\def\card{{\rm card}}

\def\pa{{\partial}}

\def\Ass{\hbox{$\mathcal{A}\!ss$}}
\def\bfk{{\mathbf k}}\def\tAss{{t\Ass}}
\def\tildeAss{\redukce{$\widetilde{\Ass}$}}

\def\calP{{\mathcal P}}

\def\ptildeAss{{p\tildeAss}{}}

\def\redukce#1{\vbox to .3em{\vss\hbox{#1}}}
\def\EE#1{{\mathcal E}^{(#1)}}

\newcommand{\ac}{\scriptstyle \text{\rm !`}}

\begin{document}

\title[Non-Koszulness and positivity]{Non-Koszulness of operads and positivity of Poincar\'e series}

\author{Vladimir Dotsenko}
\address{School of Mathematics, Trinity College, Dublin 2, Ireland}
\email{vdots@maths.tcd.ie}

\author{Martin Markl}
\address{Institute of Mathematics of the Czech Academy of Sciences, 
\v{Z}itn\'a 25, 115 67 Prague 1, The Czech Republic, and 
MFF UK, Sokolovsk\'a 83, 186 75 Prague 8, The Czech Republic}
\email{markl@math.cas.cz}

\author{Elisabeth Remm}
\address{Laboratoire de Math\'ematiques et Applications, Universit\'e de Haute Alsace, Facult\'e des
Sciences et Techniques, 4, rue des Fr\`eres Lumi\`ere, 68093 Mulhouse cedex, France}
\email{Elisabeth.Remm@uha.fr}


\begin{abstract}
We prove that the operad of mock partially associative $n$-ary algebras is not Koszul, as conjectured by the second and the third author in 2009, and utilise the Zeilberger's algorithm for hypergeometric summation to demonstrate that non-Koszulness of that operad for $n=8$ cannot be established by hunting for negative coefficients in the inverse of its Poincar\'e series. 
\end{abstract}

\maketitle

\section*{Introduction}

\subsection*{Summary of results}

Koszul duality theory for operads was developed in the seminal paper \cite{GiKa}, where it is established that among operads with quadratic relations there is an important subclass formed by Koszul operads. The category of algebras over a Koszul operad enjoys particularly nice homotopical properties. For that reason, it is important to have tools to establish whether an operad is Koszul: if it is Koszul, many questions about its algebras are answered automatically by the methods of \cite{GiKa}, if it is not Koszul, studying the homotopy category of algebras over that operad is a more unusual and challenging task. 
Currently, the most general way to establish that an operad is Koszul seems to come from operadic Gr\"obner bases \cite{BrDo,DoKh}, and the most general way to establish that an operad is not Koszul relies on a functional equation established in \cite{GiKa}. The latter equation, in slightly more modern terms, says that for a Koszul operad $\mathcal{P}$, we have
 \[
g_{\mathcal{P}}(g_{\mathcal{P}^{\ac}}(t))=t ,
 \]
where $\mathcal{P}^{\ac}$ is the Koszul dual cooperad, and $g$ is the Poincar\'e series (the generating series for the Euler characteristics of components). 

The paper \cite{GiKa} is mostly concerned with operads whose generating operations are all binary; algebras over such operads appear in applications more frequently (for example the most famous operads ever studied, those of associative algebras, commutative associative algebras, and Lie algebras, belong to that class). While it is not hard to extend Koszul duality to the case of operads whose generating operations may be of different arities (this was first done in \cite{GetJones}), or at least not all binary, some early papers on the subject ignored crucial homological degree shifts, and as a consequence some claims made in those papers were wrong. For example, the operad called the operad of $n$-ary partially associative algebras in \cite{Gn1,Gn2}, only resembles the Koszul dual operad of the operad of totally associative algebras, contrary to the claims made there. 

Recently, several examples of $n$-ary operads (that is, operads generated by operations of the same arity~$n$) were studied by the second and the third author in the papers  \cite{MaRe1,MaRe2} the first of which was circulated as a preprint back in 2009. The defining relations of those operads describe various types of ``graded $n$-associativity'' and resemble the defining relations of the operads of totally associative and partially associative $n$-ary algebras, but have different signs and homological degrees in the definition. For the latter reason, we refer to them as operads of \emph{mock} totally / partially associative $n$-ary algebras. In \cite{MaRe1,MaRe2}, some of those operads were proved to be Koszul, some of them were proved to not be Koszul, and finally, the remaining ones were conjectured to not be Koszul. In fact, it is quite easy to describe those conjecturally non-Koszul operads. Fix  $n\geq 2$. The operad $\ptildeAss^n_0$ of mock partially associative $n$-ary algebras is generated by one operation $\mu$ of arity $n$ and of degree $0$ satisfying one single relation
\[
\sum_{i=1}^n 
\mu\circ_i\mu =0.
\]
In \cite{MaRe1,MaRe2}, the operads $\ptildeAss^n_0$ are proved to be non-Koszul for $n\le 7$, and it was conjectured that they are not Koszul for all $n\ge 2$. 

The Koszul dual cooperad of $\ptildeAss^n_0$ is the cooperad $(\tAss^n_{1})^c$, whose coalgebras are mock totally coassociative coalgebras (with one operation of arity $n$ and degree $1$); that operad has an extremely simple Poincar\'e series $t-t^n+t^{2n-1}$. In this paper, we establish two results. First, we prove that the operad $\ptildeAss^n_0$ is not Koszul. For that, we establish and utilise a rather surprising combinatorial formula representing a certain element in the cobar construction of $(\tAss^n_{1})^c$ as a boundary. Second, we check that the inverse series of $t-t^n+t^{2n-1}$ for $n=8$ does not have any negative coefficients (so a positivity criterion of Koszulness based on the Ginzburg--Kapranov functional equation is not of any help); for that we make use of the Zeilberger's algorithm for hypergeometric summation.

\subsection*{Plan of the paper}

In Section \ref{sec:criteria}, we recall the key definitions needed throughout the paper.
In Section \ref{sec:nonKoszul}, we prove that the mock partially associative operad is not Koszul.
In Section \ref{sec:inverse}, we show that the result of the previous section cannot be obtained using the positivity criterion of Koszulness.

\subsection*{Funding} \thanks{The second author was supported by the
Eduard \v Cech Institute [grant P201/12/G028] and the Academy of
Sciences of the Czech Republic [grant RVO 67985840]. The final
revision was moreover supported by a grant GA \v CR [18-07776S to
M.M.] and by Pr{\ae}mium Academiae of Martin Markl.  }

\subsection*{Acknowledgements} The final draft of this paper was prepared at Max Planck Institute for Mathematics in Bonn, where the authors' stay was supported through the programme ``Higher Structures in Geometry and Physics''. The authors wish to thank MPIM for for the excellent working conditions enjoyed during their visit. The authors are also grateful to David Speyer who both provided the answer \cite{Speyer} on the \texttt{MathOverflow} website which convinced them of positivity of coefficients of the inverse series for $t-t^8+t^{15}$ and pointed out a gap in the proof of positivity in a draft version of this paper.

\section{(Non-)Koszulness and its criteria}\label{sec:criteria}

Throughout this paper, we follow the notational conventions set out in \cite{LoVa}. We briefly recall the most important notational conventions and definitions, and refer the reader to \cite[Chapter 7]{LoVa} for the details. All the results of this paper are valid for an arbitrary field $\bfk$ of characteristic zero. We use a formal symbol $s$ of homological degree $1$ to encode suspensions and de-suspensions.  

Unless otherwise specified, all operads $\mathcal{P}$ discussed in this paper are nonsymmetric, that is they are  monoids in the monoidal category of nonsymmetric collections; the monoidal structure in that latter category is denoted $\circ$. In addition, all operads are implicitly assumed reduced ($\mathcal{P}(0)=0$) and connected ($\mathcal{P}(1)\cong\bfk$). Throughout this paper, we use the abbreviation `ns' instead of the word `nonsymmetric'.  We use the notation $\mathcal{X}\cong\mathcal{Y}$ 
for isomorphisms of ns collections, and the notation $\mathcal{X}\simeq\mathcal{Y}$ for weak equivalences (quasi-isomorphisms). 

The free operad generated by a ns collection $\mathcal{X}$ is denoted $\mathcal{T}(\mathcal{X})$, the cofree (conilpotent) cooperad cogenerated by a ns collection $\mathcal{X}$ is denoted $\mathcal{T}^c(\mathcal{X})$; the former is spanned by ``tree tensors'', and has its  composition product, and the latter has the same underlying ns collection but a different structure, a decomposition coproduct. The underlying ns collection of each of those is weight graded (a tree tensor has weight $p$ if its underlying tree has $p$ internal vertices), and we denote by $\mathcal{T}(\mathcal{X})^{(p)}$ the subcollection which is the span of all tree tensors of weight $p$. Infinitesimal (partial) composition products on $\mathcal{T}(\mathcal{X})$ are denoted $\circ_i$.

\subsection{Koszul duality for quadratic (co)operads}

A pair consisting of a ns collection $\mathcal{X}$ and a subcollection $\mathcal{R}\subset\mathcal{T}(\mathcal{X})^{(2)}$ is called \emph{quadratic data}. To a choice of quadratic data one can associate the \emph{quadratic operad $\mathcal{P}=\mathcal{P}(\mathcal{X},\mathcal{R})$ with generators $\mathcal{X}$ and relations $\mathcal{R}$}, the largest quotient operad $\mathcal{O}$ of $\mathcal{T}(\mathcal{X})$ for which the composite
\[
\mathcal{R}\hookrightarrow\mathcal{T}(\mathcal{X})^{(2)}\hookrightarrow\mathcal{T}(\mathcal{X})\twoheadrightarrow\mathcal{O}
\]
is zero. Also, to a choice of quadratic data one can associate the \emph{quadratic cooperad $\mathcal{C}=\mathcal{C}(\mathcal{X},\mathcal{R})$ with cogenerators $\mathcal{X}$ and corelations $\mathcal{R}$}, the largest subcooperad $\mathcal{Q}\subset\mathcal{T}^c(\mathcal{X})$ for which the composite
 \[
\mathcal{Q}\hookrightarrow\mathcal{T}^c(\mathcal{X})\twoheadrightarrow\mathcal{T}^c(\mathcal{X})^{(2)}\twoheadrightarrow\mathcal{T}^c(\mathcal{X})^{(2)}/\mathcal{R}
 \]
is zero.

\begin{definition}[Koszul duality]
Let $(\mathcal{X},\mathcal{R})$ be a choice of quadratic data. The Koszul duality for operads assigns to an operad $\mathcal{P}=\mathcal{P}(\mathcal{X},\mathcal{R})$ its \emph{Koszul dual cooperad}
 \[
\mathcal{P}^{\ac} := \mathcal{C}(s\mathcal{X},s^2\mathcal{R}) . 
 \]  
\end{definition}

Recall that the (left) Koszul complex of a ns quadratic operad $\mathcal{P}=\mathcal{P}(\mathcal{X},\mathcal{R})$ is the ns collection $\mathcal{P}\circ\mathcal{P}^{\ac}$ equipped with a certain differential coming from a ``twisting morphism''
 \[
\varkappa\colon \mathcal{C}(s\mathcal{X},s^2\mathcal{R})\twoheadrightarrow s\mathcal{X}\to \mathcal{X}\hookrightarrow\calP(\mathcal{X},\mathcal{R}) , 
 \]
see \cite[Sec.~7.4]{LoVa} for details. 

\begin{definition}[Koszul operad]
A quadratic operad $\mathcal{P}$ is said to be \emph{Koszul} if its Koszul complex is acyclic, so that the inclusion
 \[
\bfk\cong(\mathcal{P}\circ\mathcal{P}^{\ac})(1)\hookrightarrow\mathcal{P}\circ\mathcal{P}^{\ac} 
 \]
induces an isomorphism in the homology. 
\end{definition}

For a cooperad $\mathcal{C}$, its cobar construction $\Omega(\mathcal{C})$ is, by definition, the chain complex obtained by equipping the free operad $\mathcal{T}(s^{-1}\mathcal{C})$ with the differential coming from the infinitesimal decomposition coproducts on $\mathcal{C}$. It is known \cite[Prop.~7.3.2]{LoVa} that for a Koszul operad $\mathcal{P}$ there is a weak equivalence $\Omega(\mathcal{P}^{\ac})\simeq\mathcal{P}$; that is, the cobar construction $\Omega(\mathcal{P}^{\ac})$ represents the \emph{minimal model} of $\mathcal{P}$, see \cite{Ma} for the precise definition. 

\subsection{Poincar\'e series for operads and the positivity criterion for Koszulness}

A very useful numerical invariant of a ns collection is given by its Poincar\'e series. 

\begin{definition}[Poincar\'e series]
Let $\mathcal{X}$ be a ns collection with finite-dimensional components. The generating series for Euler characteristics of components of $\mathcal{X}$ is called the \emph{Poincar\'e series} of $\mathcal{X}$ and is denoted by $g_{\mathcal{X}}(t)$:
 \[
g_{\mathcal{X}}(t) = \sum_{n\ge 0}\chi(\mathcal{X}(n))t^n .
 \]
\end{definition}

An important property of the Poincar\'e series is that it is compatible with the ns composition $\circ$.

\begin{proposition}[{\cite[Prop.~4.1.7]{GiKa}}]\label{prop:Comp}
Let $\mathcal{X}$ and $\mathcal{Y}$ be two ns collections  with finite-dimensional components. Then
 \[
g_{\mathcal{X}\circ\mathcal{Y}}(t)=g_{\mathcal{X}}(g_{\mathcal{Y}}(t)) .
 \]
\end{proposition}

\begin{corollary}\label{cor:FunEq}Let $\mathcal{P}$ be a ns operad with finite-dimensional components. 
\begin{itemize}
\item[\textnormal{(i)}] If $\mathcal{P}$ is Koszul, then 
\begin{equation}\label{eq:GKfun}
g_{\mathcal{P}}(g_{\mathcal{P}^{\ac}}(t))=t .
\end{equation}
\item[\textnormal{(ii)}] More generally, if 
 \[
(\mathcal{T}(\mathcal{E}),\pa)\simeq (\mathcal{P},0) 
 \] 
is the minimal model of $\mathcal{P}$, then
\begin{equation}
\label{Jaruska_mi_udelala_svickovou!}
g_\mathcal{P}\left(t-g_\mathcal{E}(t)\right) = t .
\end{equation}
\end{itemize}
\end{corollary}

\begin{proof}
The claim (i) follows from either the more general (ii), or from the definition of the Koszul operad using the Koszul complex. 
The claim (ii) is proved in \cite{MaRe1}; it also immediately follows from Proposition \ref{prop:Comp} and \cite[Th.~6.6.2]{LoVa}). 
\end{proof}

Equation \eqref{eq:GKfun} provides an obvious necessary condition for an operad to be Koszul. However, in many cases it is too hard to compute the Poincar\'e series of both $\mathcal{P}$ and $\mathcal{P}^{\ac}$. For that reason, the following weaker result is used in many known proofs of non-Koszulness in the available literature.

\begin{corollary}[Positivity criterion]\label{cor:GKweak}
Suppose that $\mathcal{P}$ is a quadratic ns operad with finite-dimensional components generated by operations of homological degree zero. If the compositional inverse of either of the two power series $g_{\mathcal{P}^{\ac}}(t)$ and $g_{\mathcal{P}}(t)$  has at least one negative coefficient, then $\mathcal{P}$ is not Koszul.
\end{corollary}

This criterion (or its mild variations) was utilised, for instance, in \cite{GetKa} for the ``mock Lie'' operad and the ``mock-commutative operad'', in \cite{Val} for some Manin products of operads, and in \cite{MaRe1,MaRe2} for some other mock operads of $n$-ary algebras. 

\subsection{The gap criterion for $n$-ary operads}

We fix $n\ge 2$. Suppose that $\mathcal{P}$ is an $n$-ary quadratic operad. The operad $\mathcal{P}$ has a weight grading, and so does its minimal model $(\mathcal{T}(\mathcal{E}),\pa)\simeq(\mathcal{P},0)$; we denote by $\mathcal{E}^{(p)}$ the subcollection of $\mathcal{E}$ consisting of all elements of weight $p$.
It is clear that $\mathcal{P}^{(p)}(m)=\mathcal{E}^{(p)}(m)=0$ unless $m=p(n-1)+1$ for some $p\geq0$.

\begin{definition}[{\cite[Def.~3.2]{MaRe2}}]
The minimal model $(\mathcal{T}(\mathcal{E}),\pa)$ of an $n$-ary operad has a \emph{gap of length $d \geq 1$} if there is a $q \geq 2$ such that 
\[
\mathcal{E}^{(p)} = 0 \text{ for } q \leq p \leq q+d-1
\]
while $\mathcal{E}^{(q-1)} \not= 0 \not = \mathcal{E}^{(q+d)}$.
\end{definition}

\begin{proposition}[Gap criterion, {\cite{MaRe1}}]\label{prop:Gap}
Suppose that the minimal model of a quadratic $n$-ary operad $\mathcal{P}$ has a gap of finite length. Then $\mathcal{P}$ is not Koszul.
\end{proposition}

\section{The mock partially associative operad is not Koszul}\label{sec:nonKoszul}

Let us fix  $n\geq 2$. In this section, we study the operad $\ptildeAss^n_0$ of mock partially associative $n$-ary algebras;
it is generated by one operation $\mu$ of arity $n$ and of degree $0$ satisfying one single relation
\[
\sum_{i=1}^n 
\mu\circ_i\mu =0.
\]
In \cite{MaRe1}, the weak Ginzburg--Kapranov criterion was used to establish that the operads $\ptildeAss^n_0$
are not Koszul for $n \le 7$, and it was conjectured that they are not Koszul for all $n\ge 2$. In this section we prove this conjecture:

\begin{theorem}\label{th:nonKoszul}
The operad $\ptildeAss^n_0$ is not Koszul for an arbitrary $n \geq 2$.
\end{theorem}

The proof goes as follows. From~\cite[Prop.~14]{MaRe1}, it follows that the Koszul dual cooperad of  
$\ptildeAss^n_0$ is the cooperad $(\tAss^n_{1})^c$, whose coalgebras are mock totally coassociative coalgebras (with
one operation of arity $n$ and degree $1$). From~\cite[Lemma~19]{MaRe1}, it follows that the only nonzero 
components of that latter cooperad are those of arities $1$, $n$ and $2n-1$. 

Assume that the operad~$\ptildeAss^n_0$ is Koszul, so that it coincides with the homology of the cobar construction  
$\Omega((\tAss^n_{1})^c)$. Explicitly, the cobar construction is freely generated by an operation $\mu$ of arity $n$ and degree $0$,
and an operation $\xi$ of arity $2n-1$ and degree $1$; its differential $\pa$ is given by
\[
\pa(\mu) := 0,\quad  \pa(\xi) := 
\sum_{i=1}^n 
\mu\circ_i\mu.
\] 
As usual, we will represent elements of the free operad as linear combinations of planar rooted
trees. In homological degree $0$ we have trees with $n$-ary vertices, and in degree $1$
trees with $n$-ary vertices and exactly one vertex of arity $2n-1$, which we call the \emph{fat} vertex.
The central role in the proof is played by the element $\mn$ obtained by iterated composition of $n+1$ copies $\mu$,
where each composition is at the last slot. For example, a pictorial presentation 
of $\mu^{(4)}$ is
\begin{center}
{
\unitlength=.5pt
\begin{picture}(470.00,105.00)(-160.00,120.00)
\put(110.00,140.00){\makebox(0.00,0.00){$\tecka$}}
\put(90.00,160.00){\makebox(0.00,0.00){$\tecka$}}
\put(70.00,180.00){\makebox(0.00,0.00){$\tecka$}}
\put(50.00,200.00){\makebox(0.00,0.00){$\tecka$}}
\put(110.00,140.00){\line(1,-1){20.00}}
\put(110.00,140.00){\line(0,-1){20.00}}
\put(90.00,120.00){\line(1,1){20.00}}
\put(90.00,160.00){\line(-1,-1){20.00}}
\put(90.00,140.00){\line(0,1){20.00}}
\put(70.00,180.00){\line(0,-1){20.00}}
\put(70.00,180.00){\line(-1,-1){20.00}}
\put(50.00,200.00){\line(0,-1){20.00}}
\put(50.00,200.00){\line(-1,-1){20.00}}
\put(50.00,220.00){\line(0,-1){20.00}}
\put(50.00,200.00){\line(1,-1){60.00}}
\put(140.00,140.00){{.}}
\end{picture}}
\end{center}

Computing Gr\"obner bases \cite{BrDo} of the defining relations for operads $\ptildeAss^n_0$ for small $n$, one notices that the operation $\mn$ always appears as a Gr\"obner basis element, and so it is natural to conjecture that the operation $\mn$ vanishes in any $\ptildeAss^n_0$-algebra. We establish that result below. The operation $\mn$ has weight $n+1$ and arity $n^2$, and in fact, it is not completely surprising that some unexpected vanishing result can be proved for that weight / arity. Indeed, according to \cite[Prop.~10.2.2.4]{BrDo}, the number of distinct consequences 
of weight $w$ of one quadratic relation involving one $n$-ary operation is equal to $\binom{nw-1}{w-2}$, and so for $w=n+1$ that number is equal to 
 \[
\binom{n^2+n-1}{n-1}=\frac{(n^2+n-1)!}{(n^2)!(n-1)!}=\frac{1}{n^2}\frac{(n^2+n)!}{(n^2-1)!(n+1)!}=\frac{1}{n^2}\binom{n^2+n}{n+1} ,
 \]
which is the dimension of the whole weight $n+1$ component of the corresponding free operad. Therefore, for a ``generic'' relation it would even be
likely that all tree tensors vanish individually, but since our relation is far from generic, only some partial vanishing is observed. Let us note that, in general, the Gr\"obner basis of the defining relations of the operad $\ptildeAss^n_0$ seems to be far from tractable, and finding dimensions of components of that operad is an open problem; therefore, we only establish that this operad gives an example of limitations that the positivity test has, and not determine whether the Ginzburg--Kapranov functional equation holds.
Let us introduce, only for the purposes of this section, the following:

\noindent 
{\bf Terminology.}
A {\em degree-$0$ tree\/} will mean a planar rooted tree with $n+1$
vertices of arity $n$.  A~{\em degree-$1$ tree\/} will be a planar
rooted tree with $n-1$ vertices of arity $n$ and one fat vertex. (As above, this degree refers to the degree in the cobar complex.)
With a~few obvious exceptions, by a {\em tree} 
we will mean either a degree-$0$ tree or a degree-$1$ tree.
Thus $\mn$ is a~particular example of a degree-$0$ tree.

We are going to describe a rule that divides internal edges of each tree $X$ 
into two disjoint subsets, the set $\reg(X)$ of {\em regular\/} edges
and the set $\sng(X)$ of {\em singular\/} ones. For a degree-$0$ tree~$S$
and its internal edge $e\in \edg(S)$ denote by $S/e$ the degree-$1$ tree
obtained by collapsing $e$ into a~vertex. The crucial property of this
rule will be that
\begin{equation}
\label{eq:2}
\card(\reg(S/e)) =
\begin{cases}
\card(\reg(S)) -1 &\hbox { if $e$ is regular, and}
\\
\card(\reg(S)) &\hbox { if $e$ is singular.}  
\end{cases}
\end{equation}

Given a tree $X$, we ``flatten'' it in such a~way that its  
rightmost input leg is at the same level as its root leg, resulting in a diagram of the form
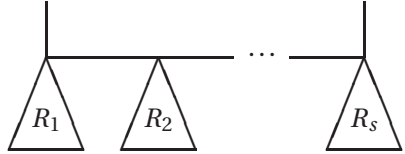
\[
{
\unitlength=.7pt
\begin{picture}(210.00,80.00)(0.00,0.00)
\thicklines
\put(190.00,10.00){\makebox(0.00,0.00)[b]{$R_s$}}
\put(80.00,10.00){\makebox(0.00,0.00)[b]{$R_2$}}
\put(20.00,10.00){\makebox(0.00,0.00)[b]{$R_1$}}
\put(136,50.00){\makebox(0.00,0.00){$\cdots$}}
\put(190.00,50.00){\line(0,1){30.00}}
\put(150.00,50.00){\line(1,0){40.00}}
\put(20.00,50.00){\line(1,0){100.00}}
\put(20.00,80.00){\line(0,-1){30.00}}
\put(190.00,50.00){\line(-2,-5){20.00}}
\put(210.00,0.00){\line(-2,5){20.00}}
\put(170.00,0.00){\line(1,0){40.00}}
\put(80.00,50.00){\line(-2,-5){20.00}}
\put(100.00,0.00){\line(-2,5){20.00}}
\put(60.00,0.00){\line(1,0){40.00}}
\put(40.00,0.00){\line(-2,5){20.00}}
\put(0.00,0.00){\line(1,0){40.00}}
\put(20.00,50.00){\line(-2,-5){20.00}}
\end{picture}}
\]
where $R_i$'s are, for $1\leq i \leq s$, planar rooted trees.  We call the result the \emph{body} of the tree~$X$. 
The \emph{soul} of a tree $X$ is obtained from its body by removing all the external legs; it is
a diagram of the form  
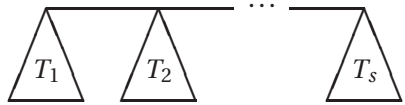
\begin{equation}
\label{dnes_vecer_Plastici}
{
\unitlength=.7pt
\begin{picture}(210.00,30.00)(0.00,60.00)
\put(0,30){
\thicklines
\put(190.00,10.00){\makebox(0.00,0.00)[b]{$T_s$}}
\put(80.00,10.00){\makebox(0.00,0.00)[b]{$T_2$}}
\put(20.00,10.00){\makebox(0.00,0.00)[b]{$T_1$}}
\put(136,50.00){\makebox(0.00,0.00){$\cdots$}}
\put(150.00,50.00){\line(1,0){40.00}}
\put(20.00,50.00){\line(1,0){100.00}}
\put(190.00,50.00){\line(-2,-5){20.00}}
\put(210.00,0.00){\line(-2,5){20.00}}
\put(170.00,0.00){\line(1,0){40.00}}
\put(80.00,50.00){\line(-2,-5){20.00}}
\put(100.00,0.00){\line(-2,5){20.00}}
\put(60.00,0.00){\line(1,0){40.00}}
\put(40.00,0.00){\line(-2,5){20.00}}
\put(0.00,0.00){\line(1,0){40.00}}
\put(20.00,50.00){\line(-2,-5){20.00}}
}
\end{picture}}
\end{equation}
\vskip 20pt
\noindent 
where $T_i$'s are trees with no external legs. Note that there is a one-to-one correspondence between
the set  $\edg(X)$ of internal edges of $X$ and the set of edges of
its soul. In other words, the soul of $X$ is the subtree of $X$ spanned by
its internal edges drawn as in~(\ref{dnes_vecer_Plastici}). 

We call an internal 
edge of $X$ \emph{singular} if it corresponds to the outgoing edge of a non-fat vertex of the soul of $X$  
with no input edge, i.e.\ when the corresponding edge in the soul looks as
\[
{
\unitlength=.5pt
\begin{picture}(0.00,20.00)(0.00,5.00)
\put(0.00,0.00){\makebox(0.00,0.00){$\tecka$}}
\put(0.00,30.00){\line(0,-1){30.00}}
\end{picture}}
\]
where $\tecka$ has no input edges (we emphasize here that horizontal
edges are not considered as input edges). All remaining internal 
edges of $X$ are called  \emph{regular}.
It is easy to see that this division of edges into regular and singular fulfills~\eqref{eq:2}. 
We believe that Figures~\ref{Jim_Cert} and~\ref{PPU} explain what we
mean; in these figures, non-fat vertices are represented by bullets
$\tecka$ and fat vertices are represented by black
squares~$\tlustatecka$, all the singular edges are 
dotted,  and all the regular ones are thick.

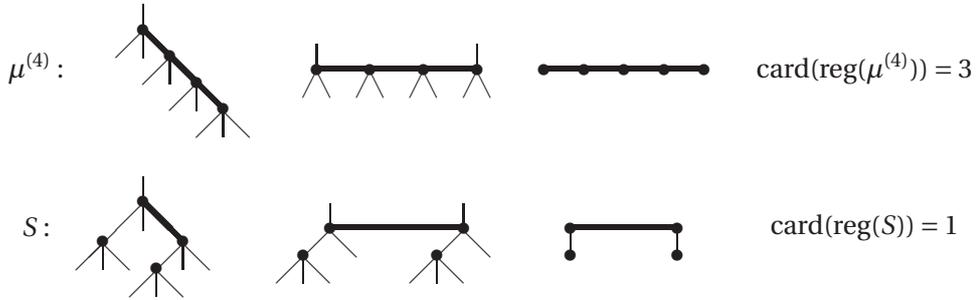
\begin{figure}[!!h]
\begin{center}
{
\unitlength=.5pt
\begin{picture}(470.00,220.00)(120.00,0.00)
\put(450.00,30.00){\makebox(0.00,0.00){$\tecka$}}
\put(450.00,50.00){\makebox(0.00,0.00){$\tecka$}}
\put(370.00,30.00){\makebox(0.00,0.00){$\tecka$}}
\put(370.00,50.00){\makebox(0.00,0.00){$\tecka$}}
\multiput(367.00,46.00)(0,-4){5}{$\cdot$}
\multiput(447.00,46.00)(0,-4){5}{$\cdot$}
{
\thicklines
\multiput(370.00,51)(0,.1){20}{\line(1,0){80.00}}
}
\put(270.00,30.00){\makebox(0.00,0.00){$\tecka$}}
\put(290.00,50.00){\makebox(0.00,0.00){$\tecka$}}
\put(190.00,50.00){\makebox(0.00,0.00){$\tecka$}}
\put(170.00,30.00){\makebox(0.00,0.00){$\tecka$}}
\put(270.00,30.00){\line(0,-1){20.00}}
\put(270.00,30.00){\line(1,-1){20.00}}
\put(270.00,30.00){\line(-1,-1){20.00}}
\put(290.00,50.00){\line(1,-1){20.00}}
\multiput(288.00,50.00)(-4,-4){5}{.}
\put(190.00,50.00){\line(1,-1){20.00}}
\put(170.00,30.00){\line(0,-1){20.00}}
\put(170.00,30.00){\line(1,-1){20.00}}
\multiput(188.00,50.00)(-4,-4){5}{.}
\put(170.00,30.00){\line(-1,-1){20.00}}
\put(290.00,50.00){\line(0,1){20.00}}
{\thicklines
\multiput(190.00,51.00)(0,.1){20}{\line(1,0){100.00}}
}
\put(190.00,70.00){\line(0,-1){20.00}}
\put(60.00,20.00){\makebox(0.00,0.00){$\tecka$}}
\put(80.00,40.00){\makebox(0.00,0.00){$\tecka$}}
\put(20.00,40.00){\makebox(0.00,0.00){$\tecka$}}
\put(50.00,70.00){\makebox(0.00,0.00){$\tecka$}}
\put(60.00,20.00){\line(1,-1){20.00}}
\put(60.00,20.00){\line(0,-1){20.00}}
\put(60.00,20.00){\line(-1,-1){20.00}}
\put(80.00,40.00){\line(1,-1){20.00}}
\put(80.00,40.00){\line(0,-1){20.00}}
\multiput(78.00,40.00)(-4,-4){5}{.}
\put(20.00,40.00){\line(0,-1){20.00}}
\put(20.00,40.00){\line(1,-1){20.00}}
\put(20.00,40.00){\line(-1,-1){20.00}}
{
\thicklines
\multiput(50.00,70.00)(.1,.1){20}{\line(1,-1){30.00}}
}
\put(50.00,70.00){\line(0,-1){20.00}}
\multiput(48.00,70.00)(-4.1,-4.1){8}{.}
\put(50.00,90.00){\line(0,-1){20.00}}
\put(440.00,170.00){\makebox(0.00,0.00){$\tecka$}}
\put(410.00,170.00){\makebox(0.00,0.00){$\tecka$}}
\put(380.00,170.00){\makebox(0.00,0.00){$\tecka$}}
\put(350.00,170.00){\makebox(0.00,0.00){$\tecka$}}
\put(300.00,170.00){\makebox(0.00,0.00){$\tecka$}}
\put(260.00,170.00){\makebox(0.00,0.00){$\tecka$}}
\put(220.00,170.00){\makebox(0.00,0.00){$\tecka$}}
\put(180.00,170.00){\makebox(0.00,0.00){$\tecka$}}
{\thicklines
\multiput(350.00,171.00)(0,.1){20}{\line(1,0){90.00}}
}
\put(260.00,170.00){\line(1,-2){10.00}}
\put(300.00,190.00){\line(0,-1){20.00}}
\put(180.00,190.00){\line(0,-1){20.00}}
\put(300.00,170.00){\line(-1,-2){10.00}}
\put(300.00,170.00){\line(1,-2){10.00}}
\put(260.00,170.00){\line(1,0){40.00}}
\put(250.00,150.00){\line(1,2){10.00}}
\put(220.00,170.00){\line(1,0){40.00}}
\put(220.00,170.00){\line(1,-2){10.00}}
\put(220.00,170.00){\line(-1,-2){10.00}}
\put(190.00,150.00){\line(-1,2){10.00}}
\put(180.00,170.00){\line(-1,-2){10.00}}
{
\thicklines
\multiput(180.00,171.00)(0,.1){20}{\line(1,0){120.00}}
}
\put(110.00,140.00){\makebox(0.00,0.00){$\tecka$}}
\put(90.00,160.00){\makebox(0.00,0.00){$\tecka$}}
\put(70.00,180.00){\makebox(0.00,0.00){$\tecka$}}
\put(50.00,200.00){\makebox(0.00,0.00){$\tecka$}}
\put(590.00,172.00){\makebox(0.00,0.00){$\card(\reg(\mu^{(4)}))=3$}}
\put(-30.00,172.00){\makebox(0.00,0.00){$\mu^{(4)}:$}}
\put(590.00,52.00){\makebox(0.00,0.00){$\card(\reg(S))=1$}}
\put(-30,52.00){\makebox(0.00,0.00){$S:$}}
\put(110.00,140.00){\line(1,-1){20.00}}
\put(110.00,140.00){\line(0,-1){20.00}}
\put(90.00,120.00){\line(1,1){20.00}}
\put(90.00,160.00){\line(-1,-1){20.00}}
\put(90.00,140.00){\line(0,1){20.00}}
\put(70.00,180.00){\line(0,-1){20.00}}
\put(70.00,180.00){\line(-1,-1){20.00}}
\put(50.00,200.00){\line(0,-1){20.00}}
\put(50.00,200.00){\line(-1,-1){20.00}}
\put(50.00,220.00){\line(0,-1){20.00}}
{
\thicklines
\multiput(50.00,200.00)(.1,.1){20}{\line(1,-1){60.00}}
}
\end{picture}}
\end{center}
\caption{\label{Jim_Cert}
Some degree-$0$ trees for $n=3$ together with their bodies and souls.}
\end{figure}
\begin{figure}[!!!h]
\begin{center}
{
\unitlength=.50pt
\begin{picture}(100.00,480.00)(230.00,-380.00)
\put(60.00,20.00){\makebox(0.00,0.00){$\tecka$}}
\put(20.00,30.00){\makebox(0.00,0.00){$\tecka$}}
\put(60.00,20.00){\line(1,-1){20.00}}
\put(60.00,20.00){\line(-1,-1){20.00}}
\put(60.00,30.00){\line(0,-1){30.00}}
\put(20.00,30.00){\line(1,-1){20.00}}
\put(20.00,10.00){\line(0,1){20.00}}
\put(20.00,30.00){\line(-1,-1){20.00}}
\put(20.00,100.00){\makebox(0.00,0.00){}}
\put(60.00,70.00){\makebox(0.00,0.00){$\tlustatecka$}}
\put(60.00,70.00){\line(1,-2){20.00}}
\put(60.00,70.00){\line(-1,-2){20.00}}
\put(60.00,70.00){\line(1,-1){40.00}}
\multiput(58.00,70.00)(0,-6){10}{.}
\multiput(58.00,70.00)(-4,-4){10}{.}
\put(60.00,110.00){\line(0,-1){40.00}}

\put(150,0){
\put(0.00,80.00){\makebox(0.00,0.00){}}
\put(80.00,20.00){\makebox(0.00,0.00){$\tecka$}}
\put(30.00,20.00){\makebox(0.00,0.00){$\tecka$}}
\put(70.00,60.00){\makebox(0.00,0.00){$\tlustatecka$}}
\put(80.00,20.00){\line(1,-1){20.00}}
\put(80.00,20.00){\line(0,-1){20.00}}
\put(80.00,20.00){\line(-1,-1){20.00}}
\put(30.00,20.00){\line(1,-1){20.00}}
\put(30.00,0.00){\line(0,1){20.00}}
\put(30.00,20.00){\line(-1,-1){20.00}}
\put(70.00,60.00){\line(1,-1){40.00}}
\multiput(68.00,60.00)(1.5,-6){8}{.}
\put(70.00,60.00){\line(-1,-4){10.00}}
\multiput(68.00,60.00)(-4,-4){10}{.}
\put(70.00,60.00){\line(1,1){40.00}}
\put(30.00,100.00){\line(1,-1){40.00}}
}

\put(300,0){
\put(0.00,80.00){\makebox(0.00,0.00){}}
\put(90.00,20.00){\makebox(0.00,0.00){$\tecka$}}
\put(50.00,20.00){\makebox(0.00,0.00){$\tecka$}}
\put(70.00,60.00){\makebox(0.00,0.00){$\tlustatecka$}}
\put(220.00,50.00){\makebox(0.00,0.00){$\card(\reg(T_1)) = 0$}}
\put(-330.00,50.00){\makebox(0.00,0.00){$T_1:$}}
\multiput(68.00,60.00)(3,-6){7}{.}
\multiput(68.00,60.00)(-3,-6){7}{.}
}

\put(0,-130){
\put(90.00,30.00){\makebox(0.00,0.00){$\tlustatecka$}}
\put(90.00,30.00){\line(1,-2){10.00}}
\put(90.00,30.00){\line(0,-1){20.00}}
\put(80.00,10.00){\line(1,2){10.00}}
\put(90.00,30.00){\line(-1,-1){20.00}}
\put(20.00,20.00){\makebox(0.00,0.00){$\tecka$}}
\put(20.00,20.00){\line(1,-1){20.00}}
\put(20.00,0.00){\line(0,1){20.00}}
\put(20.00,20.00){\line(-1,-1){20.00}}
\put(20.00,90.00){\makebox(0.00,0.00){\relax}}
\put(60.00,60.00){\makebox(0.00,0.00){$\tecka$}}
\put(60.00,60.00){\line(0,-1){40.00}}
\multiput(58.00,60.00)(-4,-4){10}{.}
\put(60.00,100.00){\line(0,-1){40.00}}
\put(90.00,30.00){\line(1,-1){20.00}}
\thicklines
\multiput(60.00,61.00)(.1,.1){18}{\line(1,-1){30.00}}
}

\put(150,-130){
\put(0.00,60.00){\makebox(0.00,0.00){\relax}}
\put(20.00,20.00){\makebox(0.00,0.00){$\tecka$}}
\put(120.00,40.00){\line(1,-3){10.00}}
\put(120.00,40.00){\line(-1,-3){10.00}}
\put(40.00,40.00){\line(1,-1){20.00}}
\put(20.00,20.00){\line(1,-1){20.00}}
\put(20.00,20.00){\line(0,-1){20.00}}
\put(20.00,20.00){\line(-1,-1){20.00}}
\multiput(38.00,40.00)(-3.9,-3.9){5}{.}
\put(120.00,41.50){\makebox(0.00,0.00){$\tlustatecka$}}
\put(40.00,40.00){\makebox(0.00,0.00){$\tecka$}}
\put(120.00,80.00){\line(0,-1){40.00}}
\put(120.00,40.00){\line(1,-1){30.00}}
\put(120.00,40.00){\line(-1,-1){30.00}}
\put(40.00,80.00){\line(0,-1){40.00}}
\thicklines
\multiput(40.00,41.00)(0,.1){20}{\line(1,0){80.00}}
}

\put(300,-130){
\put(0.00,60.00){\makebox(0.00,0.00){\relax}}
\put(40.00,0.00){\makebox(0.00,0.00){$\tecka$}}
\multiput(38.00,40.00)(0,-6){7}{.}
\put(120.00,41.50){\makebox(0.00,0.00){$\tlustatecka$}}
\put(40.00,40.00){\makebox(0.00,0.00){$\tecka$}}
\thicklines
\multiput(40.00,41.00)(0,.1){20}{\line(1,0){80.00}}
\put(220.00,40.00){\makebox(0.00,0.00){$\card(\reg(T_2)) = 1$}}
\put(-330.00,40.00){\makebox(0.00,0.00){$T_2:$}}
}

\put(0,-260){
\put(0.00,110.00){\makebox(0.00,0.00){\relax}}
\put(60.00,20.00){\makebox(0.00,0.00){$\tecka$}}
\put(80.00,40.00){\makebox(0.00,0.00){$\tecka$}}
\put(40.00,81.50){\makebox(0.00,0.00){$\tlustatecka$}}
\put(60.00,20.00){\line(1,-1){20.00}}
\put(60.00,20.00){\line(0,-1){20.00}}
\put(60.00,20.00){\line(-1,-1){20.00}}
\put(80.00,40.00){\line(1,-1){20.00}}
\put(80.00,40.00){\line(0,-1){20.00}}
\multiput(78.00,40.00)(-3.9,-3.9){5}{.}
\put(40.00,80.00){\line(1,-2){20.00}}
\put(40.00,80.00){\line(-1,-2){20.00}}
\put(40.00,80.00){\line(0,-1){40.00}}
\put(40.00,80.00){\line(-1,-1){40.00}}
\put(40.00,120.00){\line(0,-1){40.00}}
\thicklines
\multiput(40.00,81.00)(.1,.1){17}{\line(1,-1){40.00}}
}

\put(150,-260){
\put(90.00,20.00){\makebox(0.00,0.00){$\tecka$}}
\put(110.00,40.00){\makebox(0.00,0.00){$\tecka$}}
\put(30.00,41.50){\makebox(0.00,0.00){$\tlustatecka$}}
\put(90.00,20.00){\line(1,-1){20.00}}
\put(90.00,20.00){\line(0,-1){20.00}}
\put(90.00,20.00){\line(-1,-1){20.00}}
\put(110.00,80.00){\line(0,-1){40.00}}
\put(110.00,40.00){\line(1,-1){20.00}}
\multiput(108.00,40.00)(-3.9,-3.9){5}{.}
\put(30.00,40.00){\line(1,-1){30.00}}
\put(30.00,40.00){\line(1,-3){10.00}}
\put(30.00,40.00){\line(-1,-3){10.00}}
\put(30.00,40.00){\line(-1,-1){30.00}}
\put(30.00,80.00){\line(0,-1){40.00}}
\thicklines
\multiput(30.00,41.00)(0,.1){20}{\line(1,0){80.00}}
}

\put(300,-260){
\put(0.00,60.00){\makebox(0.00,0.00){\relax}}
\put(120.00,0.00){\makebox(0.00,0.00){$\tecka$}}
\multiput(118.00,40.00)(0,-6){7}{.}
\put(120.00,40.00){\makebox(0.00,0.00){$\tecka$}}
\put(40.00,41.50){\makebox(0.00,0.00){$\tlustatecka$}}
\thicklines
\multiput(40.00,41.00)(0,.1){20}{\line(1,0){80.00}}
\put(220.00,40.00){\makebox(0.00,0.00){$\card(\reg(T_3)) = 1$}}
\put(-330.00,40.00){\makebox(0.00,0.00){$T_3:$}}
}

\put(0,-390){
\put(60.00,20.00){\makebox(0.00,0.00){$\tecka$}}
\put(20.00,60.00){\makebox(0.00,0.00){$\tecka$}}
\put(40.00,40.00){\makebox(0.00,0.00){$\tlustatecka$}}
\put(60.00,20.00){\line(1,-1){20.00}}
\put(60.00,20.00){\line(0,-1){20.00}}
\put(60.00,20.00){\line(-1,-1){20.00}}
\put(40.00,40.00){\line(1,-2){10.00}}
\put(40.00,20.00){\line(0,1){20.00}}
\put(40.00,40.00){\line(-1,-2){10.00}}
\put(40.00,40.00){\line(-1,-1){20.00}}
\put(40.00,40.00){\line(0,1){0.00}}
\put(20.00,60.00){\line(0,-1){20.00}}
\put(20.00,60.00){\line(-1,-1){20.00}}
\put(20.00,90.00){\line(0,-1){30.00}}
\thicklines
\multiput(20.00,60.00)(.1,.1){17}{\line(1,-1){20.00}}
\multiput(40.00,40.00)(.1,.1){17}{\line(1,-1){20.00}}
}

\put(150,-390){
\put(80.00,30.00){\makebox(0.00,0.00){$\tlustatecka$}}
\put(140.00,30.00){\makebox(0.00,0.00){$\tecka$}}
\put(20.00,30.00){\makebox(0.00,0.00){$\tecka$}}
\put(80.00,30.00){\line(1,-1){30.00}}
\put(90.00,0.00){\line(-1,3){10.00}}
\put(80.00,30.00){\line(-1,-3){10.00}}
\put(70.00,20.00){\line(-1,-1){20.00}}
\put(80.00,30.00){\line(-1,-1){20.00}}
\put(140.00,30.00){\line(0,1){40.00}}
\put(140.00,30.00){\line(1,-1){20.00}}
\put(140.00,30.00){\line(-1,-1){20.00}}
\put(20.00,30.00){\line(1,-1){20.00}}
\put(20.00,30.00){\line(-1,-1){20.00}}
\put(20.00,30.00){\line(0,1){0.00}}
\put(20.00,70.00){\line(0,-1){40.00}}
\thicklines
\multiput(20.00,31.00)(.1,.1){17}{\line(1,0){120.00}}
}

\put(300,-390){
\put(80.00,31.50){\makebox(0.00,0.00){$\tlustatecka$}}
\put(120.00,30.00){\makebox(0.00,0.00){$\tecka$}}
\put(40.00,30.00){\makebox(0.00,0.00){$\tecka$}}
\thicklines
\multiput(40.00,31.00)(0,.1){20}{\line(1,0){80.00}}
\put(220.00,30.00){\makebox(0.00,0.00){$\card(\reg(T_4)) = 2$}}
\put(-330.00,30.00){\makebox(0.00,0.00){$T_4:$}}
}
\end{picture}}
\end{center}
\caption{
\label{PPU}
Some degree-$1$ trees for $n=3$ together with their bodies and souls. }
\end{figure}

We denote by $\edg(X)$ the set of internal edges of a tree $X$ and $e(X)$
the cardinality of this set. Notice~that
\[
e(X) =
\begin{cases}
n& \hbox { if $X$ is a degree-$0$ tree and}
\\
n-1& \hbox { if $X$ is a degree-$1$ tree.}  
\end{cases}
\]

The core of our proof of Theorem \ref{th:nonKoszul} 
is the following combinatorial lemma.

\begin{lemma}
\label{lemma}
For a degree-$1$ tree $T$ put $g = g(T) := \card(\reg(T))$ and define
\begin{equation}
\label{Jaruska_dalsi_pohovor}
\epsilon_T := (-1)^{g+n+1} g! (n-g-1)!
\end{equation}
Then 
\begin{equation}
\label{eq:4}
\pa\big( \sum_T \epsilon_T T\big) = n!\ \mu^{(n+1)}, 
\end{equation}
with the sum in the left hand side taken over all   degree-$1$ trees.
\end{lemma}

\begin{proof}
The scheme of the proof will be clearer if we
rewrite~(\ref{eq:4}) into
\begin{equation}
\label{eq:4bis}
\pa\big( \sum_T \epsilon_T T\big) = n!\big(B_1 - (-1)^n B_0\big),
\end{equation}
in which $B_1$ (resp.~$B_0$) is the sum of all degree-$0$ trees with $\sng(S)
= \emptyset$ (resp.~with $\reg(S) = \emptyset$). 
Since, by a direct inspection, $\mn$ is the
only degree-$0$ tree with no singular edge, while each degree-$0$ tree has at least
one regular edge,~(\ref{eq:4}) will be an
immediate consequence of~(\ref{eq:4bis}).

Let us prove~(\ref{eq:4bis}). 
For a degree-$0$ tree $S$ let $\pa\big( \sum_T \epsilon_T T\big)[S]$ be
the coefficient of  $S$ in
$\pa\big( \sum_T \epsilon_T T\big)$. It is clear from the definition of
the differential that
\begin{equation}
\label{eq:3}
\pa\big( \sum_T \epsilon_T T\big)[S]
=
\sum_{e \in \edg(S)} \epsilon_{S/e} =  \sum_{e \in \reg(S)} \epsilon_{S/e}
+\sum_{e \in \sng(S)} \epsilon_{S/e}.
\end{equation}
Denote $k: = \card(\reg(S))$. By~(\ref{eq:2}) one has
\[
g(S/e) =
\begin{cases}
k -1 &\hbox { if $e$ is regular, and}
\\
k &\hbox { if $e$ is singular,}  
\end{cases}
\]
therefore
\[
\epsilon_{S/e} =
\begin{cases}
 (-1)^{k+n} (k-1)!(n-k)! &\hbox { if $e$ is regular, and}
\\
(-1)^{k+n+1} k!(n-k-1)!  &\hbox { if $e$ is singular,}  
\end{cases}
\]
Notice finally that, since 
\[
\card(\reg(S)) + \card(\sng(S)) = \card(\edg(S)) = n,
\]
one has $\card(\sng(S)) = n-k$. 
Using the above calculations we verify that, for $k \not= 0,n$, 
\begin{align*}
\pa\big( \sum_T \epsilon_T T\big)[S]&= \sum_{e \in \reg(S)} (-1)^{k+n} (k-1)!(n-k)! 
+\sum_{e \in \sng(S)} (-1)^{k+n+1} k!(n-k-1)! 
\\&= k\cdot (-1)^{k+n} (k-1)!(n-k)!  +  (n-k)\cdot
(-1)^{k+n+1}k!(n-k-1)! 
=0.
\end{align*}

If $\sng(S) = \emptyset$ then $k=n$ and the second sum in the right hand
side of~(\ref{eq:3}) vanishes while the first one equals
\[
\sum_{e\in \reg(S)} (n-1)!0! = n\cdot  (n-1)!0! = n!.
\]
The case $\reg(S) = \emptyset$ is similar.
\end{proof}

\begin{example}
For $n=2$ one has five degree-$1$ trees:
\[
{
\unitlength=1.000000pt
\begin{picture}(290.00,30.00)(20.00,0.00)
\put(60,0){
\put(290.00,15.00){\makebox(0.00,0.00){$.$}}
\put(240.00,20.00){\makebox(0.00,0.00){$T_5:=$}}
\put(280.00,10.00){\makebox(0.00,0.00){$\tlustatecka$}}
\put(270.00,20.00){\makebox(0.00,0.00){$\tecka$}}
\put(280.00,10.00){\line(0,-1){10.00}}
\put(280.00,10.00){\line(-1,-1){10.00}}
{\thicklines
\multiput(270.00,20.00)(.1,.1){8}{\line(1,-1){10.00}}
}
\put(280.00,10.00){\line(1,-1){10.00}}
\put(270.00,20.00){\line(-1,-1){10.00}}
\put(270.00,30.00){\line(0,-1){10.00}}
}

\put(25,0){
\put(180.00,20.00){\makebox(0.00,0.00){$T_4:=$}}
\put(225.00,15.00){\makebox(0.00,0.00){,}}
\put(235.00,20.00){\makebox(0.00,0.00)[l]{and}}
\put(200.00,10.00){\line(1,-1){10.00}}
\put(200.00,10.00){\line(0,-1){10.00}}
\put(200.00,10.00){\makebox(0.00,0.00){$\tlustatecka$}}
\put(210.00,20.00){\makebox(0.00,0.00){$\tecka$}}
\put(210.00,30.00){\line(0,-1){10.00}}
\put(200.00,10.00){\line(-1,-1){10.00}}
\put(210.00,20.00){\line(1,-1){10.00}}
{\thicklines
\multiput(210.00,20.00)(-.1,.1){8}{\line(-1,-1){10.00}}
}
}

\put(10,0){
\put(120.00,20.00){\makebox(0.00,0.00){$T_3:=$}}
\put(170.00,15.00){\makebox(0.00,0.00){$,$}}
\put(150.00,20.00){\makebox(0.00,0.00){$\tlustatecka$}}
\put(160.00,10.00){\makebox(0.00,0.00){$\tecka$}}
\put(150.00,30.00){\line(0,-1){10.00}}
\put(150.00,20.00){\line(-1,-1){10.00}}
\put(150.00,20.00){\line(0,-1){10.00}}
\put(160.00,10.00){\line(-1,-1){10.00}}
{\thicklines
\multiput(150.00,20.00)(.1,.1){8}{\line(1,-1){10.00}}
}
\put(160.00,10.00){\line(1,-1){10.00}}
}

\multiput(88.0,19.50)(-1.8,-1.8){5}{.}
\put(80.00,10.00){\line(-1,-1){10.00}}
\put(60.00,20.00){\makebox(0.00,0.00){$T_2:=$}}
\put(110.00,15.00){\makebox(0.00,0.00){,}}
\put(90.00,20.00){\makebox(0.00,0.00){$\tlustatecka$}}
\put(80.00,10.00){\makebox(0.00,0.00){$\tecka$}}
\put(90.00,30.00){\line(0,-1){10.00}}
\put(90.00,20.00){\line(1,-1){10.00}}
\put(90.00,20.00){\line(0,-1){10.00}}
\put(80.00,10.00){\line(1,-1){10.00}}

\put(-10,0)
{
\put(0.00,20.00){\makebox(0.00,0.00){$T_1:=$}}
\put(50.00,15.00){\makebox(0.00,0.00){,}}
\put(30.00,20.00){\makebox(0.00,0.00){$\tlustatecka$}}
\put(30.00,5.00){\makebox(0.00,0.00){$\tecka$}}
\put(30.00,30.00){\line(0,-1){10.00}}
\put(30.00,5.00){\line(1,-1){10.00}}
\put(30.00,5.00){\line(-1,-1){10.00}}
\put(30.00,20.00){\line(1,-1){10.00}}
\multiput(28.50,18)(0,-2.8){5}{$\cdot$}
\put(30.00,20.00){\line(-1,-1){10.00}}
}
\end{picture}}
\]
One sees that
$\card(\reg(T_1)) = \card(\reg(T_2)) = 0$ and $\card(\reg(T_3)) =\card(\reg(T_4))  = \card(\reg(T_5)) =
1$ so, by~(\ref{Jaruska_dalsi_pohovor}), $\epsilon_{T_1} = \epsilon_{T_2}
= -1$ and $\epsilon_{T_3} = \epsilon_{T_4} = \epsilon_{T_5}
= 1$. Equation~(\ref{eq:4}) in this case reads
\[
\pa(-T_1 -T_2 + T_3 + T_4 + T_5) = 2\mu^{(3)}. 
\]
\end{example}

\begin{example}
The trees $T_1$, $T_2$ and $T_3$ in Figure~\ref{PPU} are all degree-$1$ trees $T$ 
such that $\pa(T)[S]\not= 0$ for the degree-$0$ tree $S$ in
Figure~\ref{Jim_Cert}. The tree $T_2$ appears in the left hand side
of~(\ref{eq:4})  with
coefficient $2$, the trees $T_1$ and $T_3$ with coefficients $-1$, so
indeed  $\pa ( \sum_T \epsilon_T T\big)[S] = 0$.
\end{example}

\begin{proof}[Proof of Theorem~\ref{th:nonKoszul}]
Notice first that all coefficients $\epsilon_T$
in~(\ref{Jaruska_dalsi_pohovor}) are non-zero, and denote
\[
\nu := \sum_{T} \epsilon_T T \ \hbox {(sum over all degree-$1$ trees)}.
\] 
Let us show that the degree $1$ element $c_n := \mu \circ_n \nu - \nu \circ_{n^2} \mu$ represents a nontrivial homology class of the  cobar complex $\Omega((\tAss^n_{1})^c)$. Using~\eqref{eq:4}, we verify that
\[
\pa(c_n) =\mu \circ_n  \pa (\nu) - \pa(\nu )\circ_{n^2} \mu =
n!\big(\mu \circ_n \mn - \mn \circ_{n^2} \mu\big) = 0,
\] 
so $c_n$ is indeed a cycle. 
The crucial role in proving that $c_n$ is non-homologous to zero
is played by the ``whistle-blower'' 
\[
W_n:=\mu \circ_n \big[(\cdots ((\xi \circ_{n-1}\mu)\circ_{n-2}\mu) \cdots)
\circ_1 \mu\big].
\]
For example, the whistle-blower $W_3$ is represented by the degree-$1$ tree
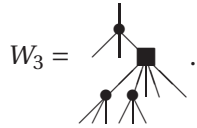
\[
{
\unitlength=.5pt
\begin{picture}(100.00,90.00)(0.00,0.00)
\put(105.00,45.00){\makebox(0.00,0.00){$.$}}
\put(60.00,20.00){\makebox(0.00,0.00){$\tecka$}}
\put(40.00,20.00){\makebox(0.00,0.00){$\tecka$}}
\put(50.00,70.00){\makebox(0.00,0.00){$\tecka$}}
\put(70.00,50.00){\makebox(0.00,0.00){$\tlustatecka$}}
\put(-10.00,50.00){\makebox(0.00,0.00){$W_3=$}}
\put(70.00,50.00){\line(0,-1){30.00}}
\put(60.00,20.00){\line(1,-2){10.00}}
\put(60.00,20.00){\line(0,-1){20.00}}
\put(60.00,20.00){\line(-1,-2){10.00}}
\put(40.00,20.00){\line(-1,-2){10.00}}
\put(40.00,20.00){\line(0,-1){20.00}}
\put(40.00,20.00){\line(-1,-1){20.00}}
\put(70.00,50.00){\line(1,-3){10.00}}
\put(70.00,50.00){\line(-1,-3){10.00}}
\put(70.00,50.00){\line(-1,-1){30.00}}
\put(70.00,50.00){\line(1,-1){30.00}}
\put(50.00,70.00){\line(1,-1){20.00}}
\put(50.00,70.00){\line(0,-1){20.00}}
\put(50.00,70.00){\line(-1,-1){20.00}}
\put(50.00,90.00){\line(0,-1){20.00}}
\end{picture}}
\]

We claim that the monomial  $W_n$ occurs in $c_n$ written as a linear
combination of monomials with
a non-trivial coefficient.
It is clear that $W_n$ cannot appear in $\nu \circ_{n^2} \mu$, since the
rightmost input of $W_n$ is the input of $\xi$, while the rightmost
inputs of all monomials constituting 
$\nu \circ_{n!} \mu $ are that of $\mu$. On the other hand, it
is clear that the degree-$1$ tree
\begin{equation}
\label{zima_prichazi}
x_n := (\cdots ((\xi \circ_{n-1}\mu)\circ_{n-2}\mu) \cdots)
\circ_1 \mu
\end{equation}
is the unique monomial such that $W_n = \mu \circ_n x_n$. 
For example, for $n=3$, $x_3$ is represented by the degree-$1$ tree
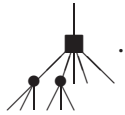
\[
{
\unitlength=.5pt
\begin{picture}(100.00,90.00)(0.00,0.00)
\put(60.00,20.00){\makebox(0.00,0.00){$\tecka$}}
\put(40.00,20.00){\makebox(0.00,0.00){$\tecka$}}
\put(70.00,50.00){\makebox(0.00,0.00){$\tlustatecka$}}
\put(105.00,45.00){\makebox(0.00,0.00){.}}
\put(70.00,50.00){\line(0,-1){30.00}}
\put(60.00,20.00){\line(1,-2){10.00}}
\put(60.00,20.00){\line(0,-1){20.00}}
\put(60.00,20.00){\line(-1,-2){10.00}}
\put(40.00,20.00){\line(-1,-2){10.00}}
\put(40.00,20.00){\line(0,-1){20.00}}
\put(40.00,20.00){\line(-1,-1){20.00}}
\put(70.00,50.00){\line(1,-3){10.00}}
\put(70.00,50.00){\line(-1,-3){10.00}}
\put(70.00,50.00){\line(-1,-1){30.00}}
\put(70.00,50.00){\line(1,-1){30.00}}
\put(70.00,50.00){\line(0,1){30.00}}
\end{picture}}
\]
The monomial $x_n$ occurs in $\nu$ with a
nontrivial coefficient, so $W_n$ appears in $c_n$ with the same
nontrivial coefficient.\footnote{Inspecting the pictorial presentation
  of $x_n$ we easily establish that this coefficient equals
  $(-1)^{n+1} (n-1)!$.}

Let us prove that $c_n$ is not a boundary. 
Assume the existence of a degree $2$ element $b_n$ such that $c_n =
\pa(b_n)$. This would in particular mean that the coefficient of $W_n$ in $\pa(b_n)$ is non-zero. The whistle-blower $W_n$ was defined in such a way that all
internal edges of the corresponding tree $W_n$ connect non-fat vertices $\tecka$ representing $\mu$
with the fat vertex $\tlustatecka$, as in the graphical representation of $W_3$ above. All
trees whose differentials may contain $W_n$ are obtained by
contracting an internal edge of $W_n$. This contraction produces a vertex
with $3n-2$ inputs, while there is no generator of the cobar complex of this arity.   
\end{proof}

\begin{remark}
The result we just proved establishes that the cobar complex $\Omega((\tAss^n_{1})^c)$ has homology classes of positive degree, at least of weight $n+2$. We do not know if that is the smallest value of weight for which non-trivial homology classes exist. It is also worth noting that our proof was using the characteristic zero assumption in a rather crucial way; it would be interesting to see if it can be relaxed.
\end{remark}

To conclude this section, let us outline an alternative proof of the fact that the operad $\ptildeAss^n_0$ is not Koszul for $n=8$ (the case of a particular interest in the following section), not relying directly on the knowledge of its
Koszul dual; we believe this proof is of independent interest. To that end, we show that the minimal model of the operad  $\ptildeAss^8_0$ has
a gap of finite length, so that Proposition \ref{prop:Gap} applies.  We begin with the following general statement.

\begin{lemma}
\label{dnes_mi_neni_dobre}
Let  $\mathcal{P}$ be a quadratic operad generated by operations of the same arity $n \geq 2$ and of the same homological degree $d$.
Then the generators of the minimal model for $\mathcal{P}$ in weight $1$, $2$ and $3$ are concentrated in homological degrees $d$,
$2d+1$ and $3d+2$, respectively.
\end{lemma}

\begin{proof}
By assumption, $\mathcal{P}=\mathcal{P}(\mathcal{X},\mathcal{R})$ with the generating collection
$\mathcal{X}$ concentrated in arity $n$ and homological degree $d$. Since $\mathcal{P}$ is quadratic, $\mathcal{R}$ must be concentrated
in arity $2n-1$ and homological degree~$2d$. The $2$-step approximation to the minimal model for $\mathcal{P}$ (not taking into
account higher syzygies) is therefore of the form
\[
\mathcal{P} \stackrel {\rho_2} \longleftarrow 
\big(\mathcal{T}(\EE1,\EE2),\pa\big),
\]
with the weight~$1$ part  $\EE1$  concentrated in arity $n$ and homological degree $d$,
and the weight~$2$ part $\EE2$ in arity $2n-1$ and homological degree
$2d+1$. The image $\pa(\EE2)$ generates the operadic ideal
of relations and  $\pa|_{\EE2}$ is a monomorphism.

The three-step approximation to the minimal model for $\mathcal{P}$ is
of the form 
\[
\mathcal{P} \stackrel {\rho_3} \longleftarrow 
\big(\mathcal{T}(\EE1,\EE2,\EE3),\pa\big),
\]
where $\pa(\EE3)$ kills the homology classes in the kernel of
$H(\rho_2)$ in weight $3$ and arity $3n-2$. 
Notice that the weight~$3$ part 
$\mathcal{T}(\EE1,\EE2)^{(3)}$ of $\mathcal{T}(\EE1,\EE2)$ decomposes as
\[
\mathcal{T}(\EE1,\EE2)^{(3)} = \mathcal{T}(\EE1)^{(3)} \oplus 
\mathcal{T}(\EE1,\EE2)^{(1,1)}, 
\]
where $\mathcal{T}(\EE1,\EE2)^{(1,1)}$ is the subspace of
$\mathcal{T}(\EE1,\EE2)$ spanned by infinitesimal compositions of one element of
$\EE1$ with one element of $\EE2$.
The kernel of $H(\rho_2)(3n-2)$ is therefore 
concentrated in homological degrees  $3d$ and  $3d+1$.
Observing that
\hbox {$H_{3d}(\rho_2)(3n-2)$} is an isomorphism
\[
H_{3d}\big(\mathcal{T}(\EE1,\EE2),\pa\big)(3n-2)
\cong \mathcal{T}(\EE1)/(\pa \EE2) (3n-2) \cong \mathcal{P}(3n-2),
\] 
we conclude that the only elements to be killed by $\EE3$ are of degree 
$3d+1$. This finishes the proof.
\end{proof}

\begin{remark}
Using methods of \cite{DoKh2}, it is possible to prove a stronger version of Lemma \ref{dnes_mi_neni_dobre} stating that for any quadratic operad $\mathcal{P}$ (with generators of any arities and homological degrees), the  $k$-th Quillen homology of $\mathcal{P}$ is concentrated in weight $k$ for $k\le 3$. 
\end{remark}

The proof of non-Koszulness now goes as follows. Numerical calculations using Gr\"obner bases for operads find the initial terms of the Poincar\'e series for  $\ptildeAss^8_0$ as
\[
t + t^8 + 7t^{15} + 69 t^{22} + 790 t^{29} + 9842 t^{36} + \cdots. 
\]
Using Corollary \ref{cor:FunEq} (ii),  one calculates that the Poincar\'e series
for the generators of the minimal model of  $\ptildeAss^8_0$ is 
\[
t + t^8 + t^{15} + 0 t^{22} + 0 t^{29} + 0 t^{36} + \cdots .
\]
We see that the Euler characteristic $\chi(\EE3)$ of the space of generators of the minimal model for $\ptildeAss^8_0$ in arity $22$ vanishes. By
Lemma~\ref{dnes_mi_neni_dobre}, $\EE3$ is concentrated in degree $2$, so the vanishing of $\chi(\EE3)$ implies
that $\EE3 = 0$.  Meanwhile, analyzing the proof of Theorem \ref{th:nonKoszul}, we see that in fact we did not use the Koszul duality as such:
in this proof, $\Omega((\tAss^n_{1})^c)$ may be replaced by the two-step approximation to the minimal model of  $\ptildeAss^n_0$. Therefore, the 
two-step approximation to the minimal model is not acyclic in positive degrees, and the minimal model must have a generator of higher arity, so by Proposition \ref{prop:Gap}, the operad $\ptildeAss^8_0$ is not Koszul. 

\section{The positivity criterion of Koszulness is not decisive for the operad $\ptildeAss^8_0$}\label{sec:inverse}

In this section, we consider the possibility of using the positivity criterion of Koszulness for the operad $\ptildeAss^n_0$.  Since the Koszul dual of this operad is a very simple cooperad $\left(\tAss^n_{1}\right)^c$, it is natural to try to prove non-Koszulness by establishing that the compositional inverse of the Poincar\'e series of the latter cooperad has negative coefficients. This works for $n\le 7$, as shown in \cite{MaRe1,MaRe2}, but it turns out that for $n=8$ the inverse series does not have any negative coefficients, which we demonstrate below. For an idea of a different proof using the saddle point method, see~\cite{Speyer}. 

We first recall a classical result on inversion of power series. To state it, we use, for a  formal power series $F(t)$, the notation $\left[t^k\right]F(t)$ for the coefficient of $t^k$ in $F(t)$, and the notation $F(t)^{\langle -1\rangle}$ for the compositional inverse of $F(t)$ (if that inverse exists).

\begin{proposition}[Lagrange's inversion formula {\cite[Sec.~5.4]{Stan2}}]\label{prop:LIT}
Let $f(t)$ be a formal power series without a constant term and with a nonzero coefficient of $t$.  Then $f(t)$ has a compositional inverse, and
 \[
\left[t^k\right]f(t)^{\langle -1\rangle}=\frac1k \left[u^{k-1}\right]\left(\frac{u}{f(u)}\right)^k .
 \]
\end{proposition}

Let us now prove the main result of this section. Namely, we show that the compositional inverse of the power series $g_{\left(\tAss^8_{1}\right)^c}(t)$ has nonnegative coefficients, and hence the positivity criterion of Corollary~\ref{cor:GKweak} cannot be used to establish the non-Koszulness of the operad $\ptildeAss^n_0$.  

\begin{theorem}
The compositional inverse of the power series 
 \[
g_{\left(\tAss^8_{1}\right)^c}(t)=t-t^8+t^{15}
 \]
is of the form $t\,h(t^7)$, where $h$ is a power series with positive coefficients. 
\end{theorem}

\begin{proof}
First, let us recall the usual argument explaining the form of the inverse series. By Proposition \ref{prop:LIT}, we have 
 \[
\left[t^k\right](t-t^8+t^{15})^{\langle -1\rangle}=\frac1k \left[u^{k-1}\right]\left(\frac{u}{u-u^8+u^{15}}\right)^k=\frac1k \left[u^{k-1}\right]\left(\frac{1}{1-u^7+u^{14}}\right)^k ,
 \]
and the coefficients on the right vanish unless $k=7n+1$, so the inverse series is of the form $t\,h(t^7)$, where $h$ is some formal power series. 

Let us start the asymptotic analysis of the coefficients of the series $h(t)$. 

\begin{lemma}\label{lm:radius}
The radius of convergence of $h(t)$ is equal to $\frac{21^7}{5^{15}}$.
\end{lemma}

\begin{proof}
The radius of convergence of  $(t-t^8+t^{15})^{\langle -1\rangle}$ is equal to the maximal $r$ for which the inverse function of $t-t^8+t^{15}$ is analytic for the arguments whose modulus is smaller than~$r$. It is obvious that such $r$ is the value of $t-t^8+t^{15}$ at the modulus of the smallest zero of 
 \[
\left(t-t^8+t^{15}\right)'=1-8t^7+15t^{14}=(1-3t^7)(1-5t^7) .
 \]
As the latter modulus is manifestly $\frac{1}{\sqrt[7]5}$, the radius of convergence of the inverse series is 
 \[
\frac{1}{\sqrt[7]5}\left(1-\frac15+\frac1{25}\right)=\frac{1}{\sqrt[7]5}\frac{21}{25} .
 \]
Now, as $(t-t^8+t^{15})^{\langle -1\rangle}=t\,h(t^7)$, the radius of convergence of $h(t)$ is equal to $\left(\frac{1}{\sqrt[7]5}\frac{21}{25}\right)^7=\frac{21^7}{5^{15}}$.
\end{proof}

\begin{lemma}
The $n$-th coefficient of $h(t)$ is equal to 
 \[
a_n=\frac{1}{7n+1}\sum_{k=0}^{\lfloor n/3\rfloor}(-1)^k\binom{7n+k}{k}\binom{7n+1}{n-3k} .
 \]
\end{lemma}

\begin{proof}
Continuing the computation that utilizes the Lagrange's inversion formula, we see that the $n$-th coefficient of $h$, or equivalently the coefficient of $t^{7n+1}$ of $(t-t^8+t^{15})^{\langle -1\rangle}$, is equal to 
 \[
\frac1{7n+1} \left[u^{7n}\right]\left(\frac{1}{1-u^7+u^{14}}\right)^{7n+1}=\frac1{7n+1} \left[v^{n}\right]\left(\frac{1}{1-v+v^2}\right)^{7n+1} 
 \] 
It remains to note that \[\frac1{1-v+v^2}=\frac{1+v}{1+v^3} ,\] so 
 \[
\left(\frac{1}{1-v+v^2}\right)^{7n+1} = \left(\frac{1+v}{1+v^3}\right)^{7n+1}=\left(\sum_{i\ge 0}\binom{7n+1}{i} v^i\right)\left(\sum_{j\ge 0}(-1)^j\binom{7n+j}{j}v^{3j}\right),
 \]
therefore the coefficient of $t^{7n+1}$ is given by the requested formula  
 \[
a_n=\frac{1}{7n+1}\sum_{k=0}^{\lfloor n/3\rfloor}(-1)^k\binom{7n+k}{k}\binom{7n+1}{n-3k} .
 \] 
\end{proof}

The expression $a_n$ is given by the formula which is a sum of ``hypergeometric'' terms, we see that Zeilberger's algorithm \cite[Ch.~6]{AeqB} applies. We used the interface to it provided by the \texttt{sumrecursion} function of \textsf{Maple}; this function implements the Koepf's version of Zeilberger's algorithm \cite[Ch.~7]{Koepf}. This function instantly informs us that the sequence $\{a_n\}$ is a solution to a rather remarkable three term finite difference equation 
\begin{equation}\label{eq:3term}
s_0(n)x_n-s_1(n)x_{n-1}+s_2(n)x_{n-2}=0 , 
\end{equation}
where
\begin{align*}
s_0(n)&=2187\left(\prod_{k=0}^{13}(7n+1-k)\right) (215870371n^6-1295222226n^5+2527684225n^4\\
&-
658627050n^3-3846578936n^2+4812446376n-1760658480) ,
\end{align*}
\begin{align*}
s_1(n)&=\left(\prod_{k=0}^{6}(7n-6-k)\right)(13362081892033179314n^{13}-126939777974315203483n^{12}\\
&+485734175892096120376n^{11}-848711700458546819207n^{10}+123881005609280551032n^9\\
&+2596574853470043847011n^8-6061259307194791053272n^7+7497470293244974003099n^6\\
&-5912167336650049878706n^5+3092269284168816801572n^4-1062333018859963548504n^3\\
&+228076143949070673408n^2-27319025166066426240n+1361946602938521600) ,
\end{align*}
and 
\begin{align*}
s_2(n)&=15(15n-14)\left(\prod_{k=0}^{12}(15n-16-k)\right)(215870371n^6-710371340n^4\\
&+ 817295010n^3-370521431n^2+73255350n-5085720) .
\end{align*}

The polynomials $s_i(n)$ are of the same degree $20$, and so our equation is of the type considered by Poincar\'e in \cite{Po}. Namely, in \cite[\S2]{Po} linear finite difference equations of order $k$
 \[
s_0(n)x_n+s_1(n)x_{n-1}+\cdots+s_k(n)x_{n-k}=0 
 \]
are considered, with the additional assumption that $s_0(n)$, \ldots, $s_k(n)$ are polynomials of the same degree $d$. To such an equation, one associates its characteristic polynomial
 \[
\chi(t)=\alpha_0t^k+\alpha_1t^{k-1}+\cdots+\alpha_k=0 ,
 \]
where $\alpha_i$ is the coefficient of $t^d$ in $s_i(n)$. If the absolute values of the complex roots of $\chi(t)$ are pairwise distinct, then for any solution $\{a_n\}$ to our equation, the limit 
 \[
\lim_{n\to\infty}\frac{a_n}{a_{n-1}}
 \]
exists and is equal to one of the roots of $\chi(t)$. Usually, that root will be the one which is maximal in absolute value. The particular case when the root is the minimal in absolute value is the hardest both for computations and for the qualitative analysis of the asymptotic behaviour, since in this case the corresponding solution is unique up to proportionality, and so the situation is not stable under small perturbations.  In our case the polynomial $\chi(t)$ is
 \[
320194878522045287813073t^2-11004249007610680591789502t+94528316575149444580078125 ,
 \]
and its roots are 
 \[
\lambda_-=\frac{30517578125}{1801088541}\approx 16.943963 \quad \text{and}\quad \lambda_+=\frac{14348907}{823543}\approx 17.423385, 
 \]
so Poincar\'e theorem applies. In fact, $\lambda_-=\frac{5^{15}}{21^7}$,
so by Lemma \ref{lm:radius} it is equal to the inverse of the radius of convergence of $h(t)$. By the usual ratio formula for the radius of convergence, we see that 
 \[
\lim_{n\to\infty}\frac{a_{n}}{a_{n-1}}=\lambda_- .
 \]

Let us consider the auxiliary sequence $\{b_n\}$ satisfying the same finite difference equation \eqref{eq:3term} and the initial conditions $b_0=0$, $b_1=1$. 

\begin{lemma}
All terms of the sequence $\{b_n\}$ are positive for $n>0$, and we have 
 \[
\lim_{n\to\infty}\frac{b_n}{b_{n-1}}=\lambda_+.
 \]
\end{lemma}

\begin{proof}
First, let us show that for all $n\ge 50$ we have 
\begin{equation}\label{eq:growth}
\frac{b_n}{b_{n-1}}\ge C ,
\end{equation}
where $C=\frac{b_{50}}{b_{49}}\approx 16.9452857$. This is easy to see by induction on $n$. First, for $n=50$, the statement is obvious. Next, if we suppose that it is true for all values less than the given $n$, we have
 \[
\frac{b_n}{b_{n-1}}=\frac{s_1(n)}{s_0(n)}-\frac{s_2(n)b_{n-2}}{s_0(n)b_{n-1}}>\frac{s_1(n)}{s_0(n)}-\frac{s_2(n)}{s_0(n)C},
 \]
and so it suffices to show that 
 \[
\frac{s_1(n)}{s_0(n)}-\frac{s_2(n)}{s_0(n)C}> C .
 \]
It is easy to check, using computer algebra software, that all the roots of the polynomial $s_0(n)$ are less than $2$, so this polynomial assumes positive values in the given range. Thus, the above inequality is equivalent to 
 \[
0>C^2s_0(n)-Cs_1(n)+s_2(n) .
 \] 
Using computer algebra software, we find that the latter expression is a polynomial in $n$ with the negative leading coefficient and the largest root approximately equal to $24.69$, so the step of induction is proved. We can also check directly that $b_n>0$ for all $0<n<50$, which then implies that $b_n>0$ for all $n>0$.  
Also, by Poincar\'e Theorem, the limit of the ratio $\frac{b_n}{b_{n-1}}$ as $n\to\infty$ is equal to either $\lambda_-$ or $\lambda_+$. However, 
$16.9452857>16.944>\lambda_-$, so the inequality \eqref{eq:growth} shows that the first of the two alternatives is impossible. Hence, the limiting value is $\lambda_+$.  
\end{proof}

Our results thus far imply  that $\lim\limits_{n\to\infty}\frac{a_n}{b_n}=0$, as
 \[
\frac{a_{n+1}}{b_{n+1}}=\frac{a_n}{b_n}\frac{\frac{a_{n+1}}{a_n}}{\frac{b_{n+1}}{b_n}},
 \]
and so $\frac{a_{n+1}}{b_{n+1}}$ is a multiple of $\frac{a_n}{b_n}$ by a factor close to $\frac{\lambda_-}{\lambda_+}<1$ for large~$n$, and thus our sequence can be bounded from above by a geometric sequence with a zero limit.

Now it is easy to complete the proof. We note that 
\begin{multline*}
\frac{a_n}{b_n}-\frac{a_{n-1}}{b_{n-1}}=\frac{s_1(n)a_{n-1}-s_2(n)a_{n-2}}{s_1(n)b_{n-1}-s_2(n)b_{n-2}}-\frac{a_{n-1}}{b_{n-1}}=\\
\frac{(s_1(n)a_{n-1}-s_2(n)a_{n-2})b_{n-1}-(s_1(n)b_{n-1}-s_2(n)b_{n-2})a_{n-1}}{(s_1(n)b_{n-1}-s_2(n)b_{n-2})b_{n-1}}=\\
\frac{s_2(n)(a_{n-1}b_{n-2}-a_{n-2}b_{n-1})}{s_0(n)b_nb_{n-1}}=
\frac{s_2(n)b_{n-2}}{s_0(n)b_n}\left(\frac{a_{n-1}}{b_{n-1}}-\frac{a_{n-2}}{b_{n-2}}\right)
\end{multline*}
All roots of the polynomial $s_2(n)$ are less than $2$ as well, so for $n\ge 3$ the sign of $\frac{a_n}{b_n}-\frac{a_{n-1}}{b_{n-1}}$ is the same as the sign of $\frac{a_{n-1}}{b_{n-1}}-\frac{a_{n-2}}{b_{n-2}}$, and hence the same as the sign of 
\[\frac{a_2}{b_2}-\frac{a_1}{b_1}=-\frac{77813}{276830}<0 .\] Thus, $\left\{\frac{a_n}{b_n}\right\}$ is a strictly decreasing sequence.  For a decreasing sequence with the limit zero, all terms must be positive, and hence $a_n$ is positive for all $n>0$. 
\end{proof}

\bibliographystyle{amsplain}
\providecommand{\bysame}{\leavevmode\hbox to3em{\hrulefill}\thinspace}

\end{document}